\documentclass[12pt]{amsart}

\usepackage[mathscr]{eucal}
\usepackage{amscd}
\usepackage{amsfonts}
\usepackage{amsmath, amsthm, amssymb}
\usepackage{latexsym}
\usepackage[dvips]{graphics}
\usepackage[dvipdfm]{graphicx}

\theoremstyle{plain} 
\newtheorem{thm}{Theorem}
\newtheorem{cor}[thm]{Corollary}
\newtheorem{lem}[thm]{Lemma}
\newtheorem{prop}[thm]{Proposition}

\theoremstyle{definition}
\newtheorem{defn}[thm]{Definition}
\newtheorem{ex}[thm]{Example}

\theoremstyle{remark}

\numberwithin{equation}{section}

\title{On level 1 cyclotomic KLR algebras of type $A_{n}^{(1)}$}

\author{Masahide KONISHI} 

\address{
\begin{flushleft}
         \hspace{0.3cm}  Graduate School of Mathematics\\
         \hspace{0.3cm}  Nagoya University \\
         \hspace{0.3cm}  Frocho, Chikusaku, Nagoya 464-8602 JAPAN\\
\end{flushleft}
}
\email{m10021t@math.nagoya-u.ac.jp} 
 
\begin{document}

\maketitle

\begin{abstract}

Cyclotomic KLR algebras are defined by fixing $\alpha$ and $\Gamma$, 
two weights on vertices of a quiver. 
We set a quiver $A_{n}^{(1)}$ type and fix $\alpha$ and $\Gamma$ in a special (but essential) case, 
and then show that there are systematic changes of structures. 

\end{abstract}


\section{Introduction}
Khovanov-Lauda-Rouquier algebra (KLR algebra for short) is defined by Khovanov and Lauda,
and independently Rouquier in 2008. 
Generators and relations are obtained from a quiver $\Gamma$ and a weight $\alpha$ on its vertices. 
We can regard generators as concatenation of such diagrams : 

\hspace{4.5cm}
\includegraphics[width=6cm,angle=0,keepaspectratio,clip]{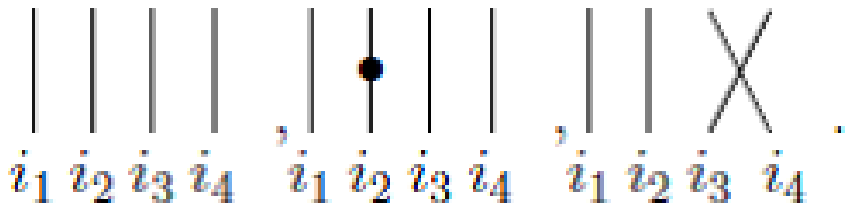}

An another weight $\Lambda$ on vertices of $\Gamma$ defines a cyclotomic ideal. 
We call a quotient of the KLR algebra by the cyclotomic ideal a cyclotomic KLR algebra. 
And the size of $\Gamma$ is called a level of a cyclotomic KLR algebra.

For each integer $n>1$, we set a quiver $A_{n}^{(1)}$ 
as its vertices are $\left\{ 0,1,2,\cdots ,n-1 \right\}$, 
and its arrows are from $i$ to $i+1$ (also $n-1$ to $0$).  
And fix $\alpha = \displaystyle \sum_{i {\rm : vertex}} \alpha_{i}$. 
From the symmetry of $A_{n}^{(1)}$, it suffices to consider the case $\Lambda = \Lambda_{0}$ 
for the level 1 cyclotomic KLR algebra. 

Our aim is to describe some properties of these algebras; 
the number of primitive idempotents, systematic changes of structures for $n$, the dimension. 

\section{Preliminaries}

Through this section, $K$ is a field and $I_{n}$ is a set consisting all of permutations of $(0,1,\cdots ,n-1)$. 

\begin{defn}
A KLR algebra $R_{n}(\alpha )$ is an algebra obtained by following generators and relations. 
\footnote{In general, there are more relations. }
\begin{itemize}
\item generators$:\left\{ {\bf e}({\bf i}) \verb@|@ {\bf i}\in I_{n} \right\} \cup 
\left\{ y_{1},\cdots,y_{n} \right\} \cup \left\{ \psi_{1},\cdots,\psi_{n-1}\right\} $
\item relations$:$\\
${\bf e}({\bf i}){\bf e}({\bf j})=\delta_{{\bf i},{\bf j}}{\bf e}({\bf i})$, \\
$\displaystyle \sum_{{\bf i}\in I_{n}} {\bf e}({\bf i})=1$, \\
$y_{k}{\bf e}({\bf i})={\bf e}({\bf i})y_{k}$, \\
$\psi_{k}{\bf e}({\bf i})={\bf e}({\bf s_{k}\cdot i})\psi_{k}$, \\
$y_{k}y_{l}=y_{l}y_{k}$, \\
$\psi_{k}y_{l}=y_{l}\psi_{k} \ (l\not =k,k+1)$, \\
$\psi_{k}\psi_{l}=\psi_{l}\psi_{k} \ (\vert k-l \vert >1)$, \\
$\psi_{k}y_{k+1}{\bf e}({\bf i})= y_{k}\psi_{k}{\bf e}({\bf i})$, \\
$y_{k+1}\psi_{k}{\bf e}({\bf i})= \psi_{k}y_{k}{\bf e}({\bf i})$, \\
$\psi_{k}^{2}{\bf e}({\bf i})= \left\{
\begin{array}{cccc}
{\bf e}({\bf i}) & (i_{k}\not \leftrightarrow i_{k+1}) \\
(y_{k+1}-y_{k}){\bf e}({\bf i}) & (i_{k}\rightarrow i_{k+1}) \\
(y_{k}-y_{k+1}){\bf e}({\bf i}) & (i_{k}\leftarrow i_{k+1}) \\
(y_{k+1}-y_{k})(y_{k}-y_{k+1}){\bf e}({\bf i}) & (i_{k}\leftrightarrow i_{k+1}) \\
\end{array}
\right.$, \\
$\psi_{k}\psi_{k+1}\psi_{k}{\bf e}({\bf i})= \psi_{k+1}\psi_{k}\psi_{k+1}{\bf e}({\bf i}) . \\$
\end{itemize}
\end{defn}

The three generators are respectively coresponding to the three diagrams in section 1. 
A multiplication of two generators are obtained as a concatenation of two diagrams 
(but if the colors of connecting part are different, it becomes 0). 
Since using diagrams is very useful to show some properties, 
we rewrite some relations to diagrams :

\noindent
\includegraphics[width=12cm,angle=0,keepaspectratio,clip]{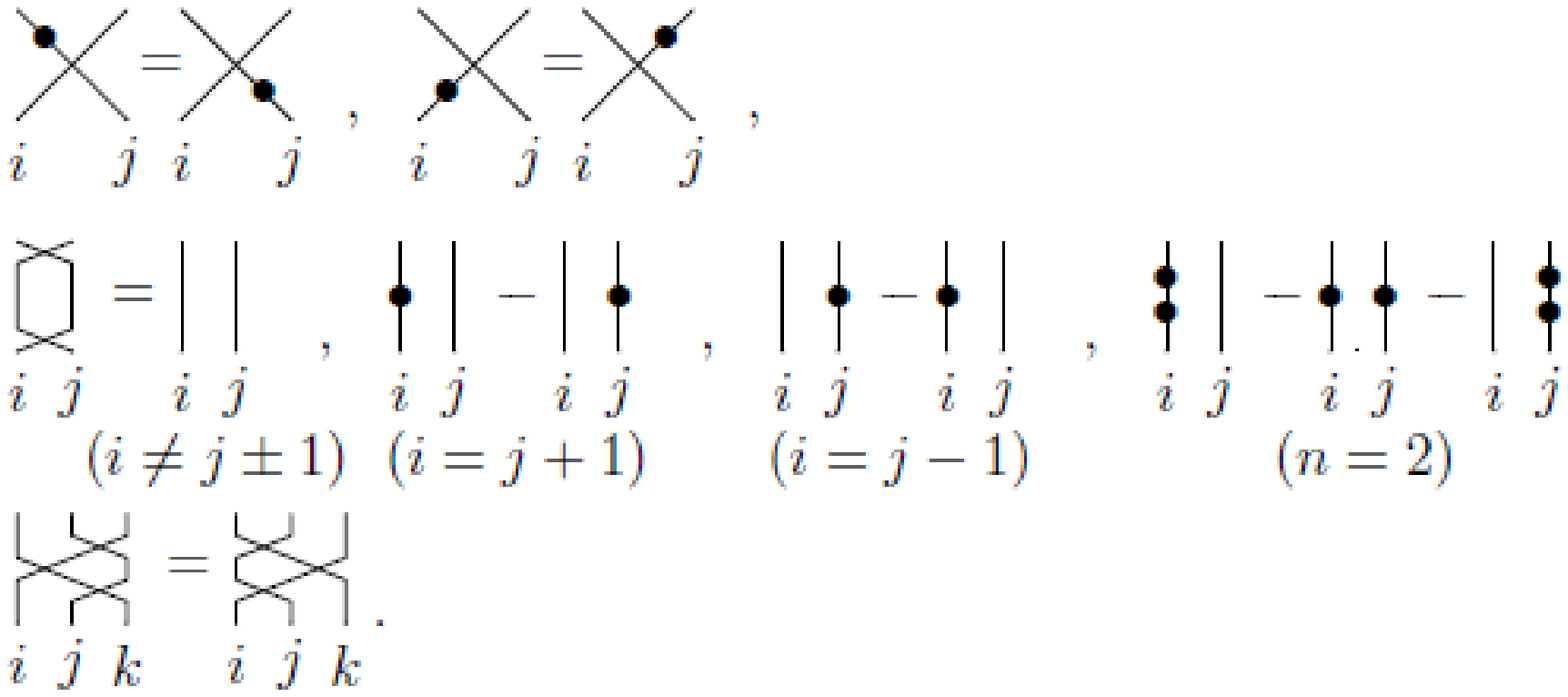}

A cyclotomic ideal $I^{\Lambda}$ and a cyclotomic KLR algebra $R_{n}$ are defined as follows.

\begin{defn}
A cyclotomic ideal $I^{\Lambda}$ is generated by 
\begin{center}
$\left\{ y_{1}{\bf e}({\bf i}) \verb@|@ {\bf i} \in I_{n} ,i_{1}=0 \right\} \cup \left\{ {\bf e}({\bf i}) \verb@|@ {\bf i} \in I_{n} ,i_{1}\neq 0 \right\}$.
\end{center}
Denote $R_{n}$ for corresponding cyclotomic KLR algebra $R_{n}(\alpha) / I^{\Lambda}$.
\end{defn}

\section{Properties}

In this section, we describe some properties of $R_{n}$. 
We need some representation theoretical facts written in the next section for proof. 

\begin{thm}\label{prop:pidem}
The number of ${\bf i} \in I_{n}$ satisfying ${\bf e}({\bf i}) \neq 0$ is exactly $2^{n-2}$. 
Moreover, the set consisting all of such ${\bf e}({\bf i})$s is complete set of primitive orthogonal idempotents. 
\end{thm}
\begin{proof}
Fix $n$. We show there are at most $2^{n-2}$ ${\bf i}$s satisfying ${\bf e}({\bf i}) \neq 0$ 
by constructing ${\bf i}$ from $i_{1}$ to $i_{b}$ avoiding ${\bf e}({\bf i})=0$. 
The rest part is proved in next section. 

In the case of $n=2$, there is only $(0,1)$. 

In the case of $n>2$, at first $i_{1}$ must be $0$ from the definition of the cyclotomic ideal. 
Next, $i_{2}$ must be $1$ or $n-1$ which are neighborhood of $0$ in the quiver.
If not, we obtain
\begin{eqnarray*}
{\bf e}((0,i_{2},\cdots )) & = & \psi_{1}^{2} {\bf e}((0,i_{2},\cdots )) \\
& = & \psi_{1} {\bf e}((i_{2},0,\cdots )) \psi_{1} \\
& = & 0 .
\end{eqnarray*}
We can write this equation by using diagrams as follows :

\hspace{5cm}
\includegraphics[width=5cm,angle=0,keepaspectratio,clip]{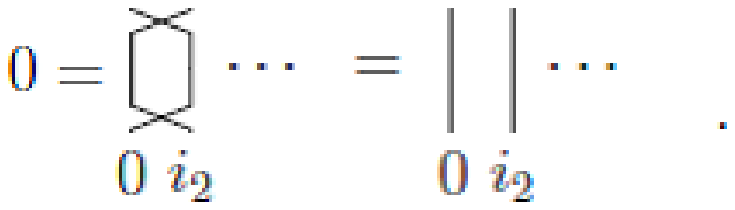}

We must keep taking one of the two neighborhoods for $i_{k}(2<k<n-1)$.
If not, ${\bf e}({\bf i})=0$ from following equation :

\hspace{5mm}
\includegraphics[width=14cm,angle=0,keepaspectratio,clip]{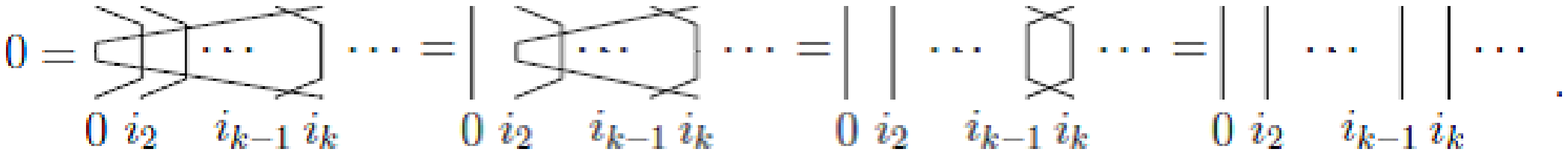}

At last, we can set the rest number for $i_{n}$. 
Then we can obtain $2^{n-2}$ ${\bf i}$s constructed by using above method.
\end{proof}

\begin{prop}\label{prop:nilp}
Let ${\bf e}({\bf i})\not =0$ in $R_{n}$. Then these properties hold :
\begin{itemize}
\item[$(a)$]$y_{k}{\bf e}({\bf i})=0 \ (1\leq k < n) $, 
\item[$(b)$]$y_{n}^{2}{\bf e}({\bf i})=0$, 
\item[$(c)$]$y_{n}{\bf e}({\bf i})\not=0$. 
\end{itemize}
\end{prop}
\begin{proof}
$(c)$ will be proved in next section. 

In the case of $n=2$, $(a)$ is by definition, $(b)$ follows by expanding $\psi {\bf e}(0,1)\psi$. 

In the case of $n>2$, we prove $(a)$ by induction for k. 

For $k=1$, $y_{k}{\bf e}({\bf i})=0$ from definition.

We show $y_{k}{\bf e}({\bf i})=0$ for $k<n$. 
By proof of Thm.3, there is unique $1\leq l<k$ such that $i_{k}$ and $i_{l}$ are neighborhoods. 
Using $y_{l}{\bf e}({\bf i})=0$ by assumption of induction, 
we obtain $y_{k}{\bf e}({\bf i})=0$ from following equation :

\includegraphics[width=14cm,angle=0,keepaspectratio,clip]{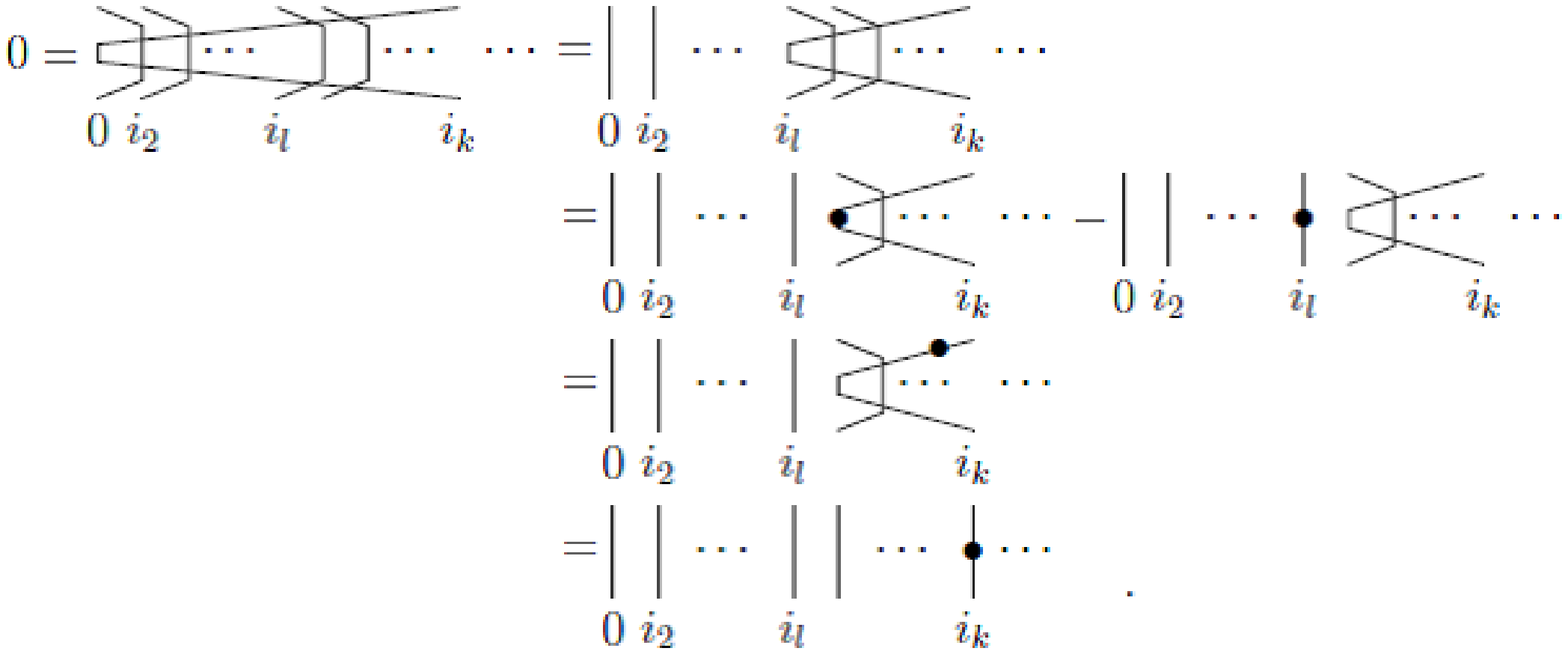}

We assume $i_{l}\rightarrow i_{k}$ in this equation, 
but if $i_{l}\leftarrow i_{k}$ the difference is only signs. Therefore $(a)$ follows. 

Using the same method, since $y_{k}{\bf e}({\bf i})=0$ for $k<n$ and 
there are two neighborhoods $i_{l}$,$i_{m}$ $(1\leq l<m<n)$ of $i_{n}$, 
we obtain $y_{n}^{2}{\bf e}({\bf i})=0$ as follows :

\includegraphics[width=14cm,angle=0,keepaspectratio,clip]{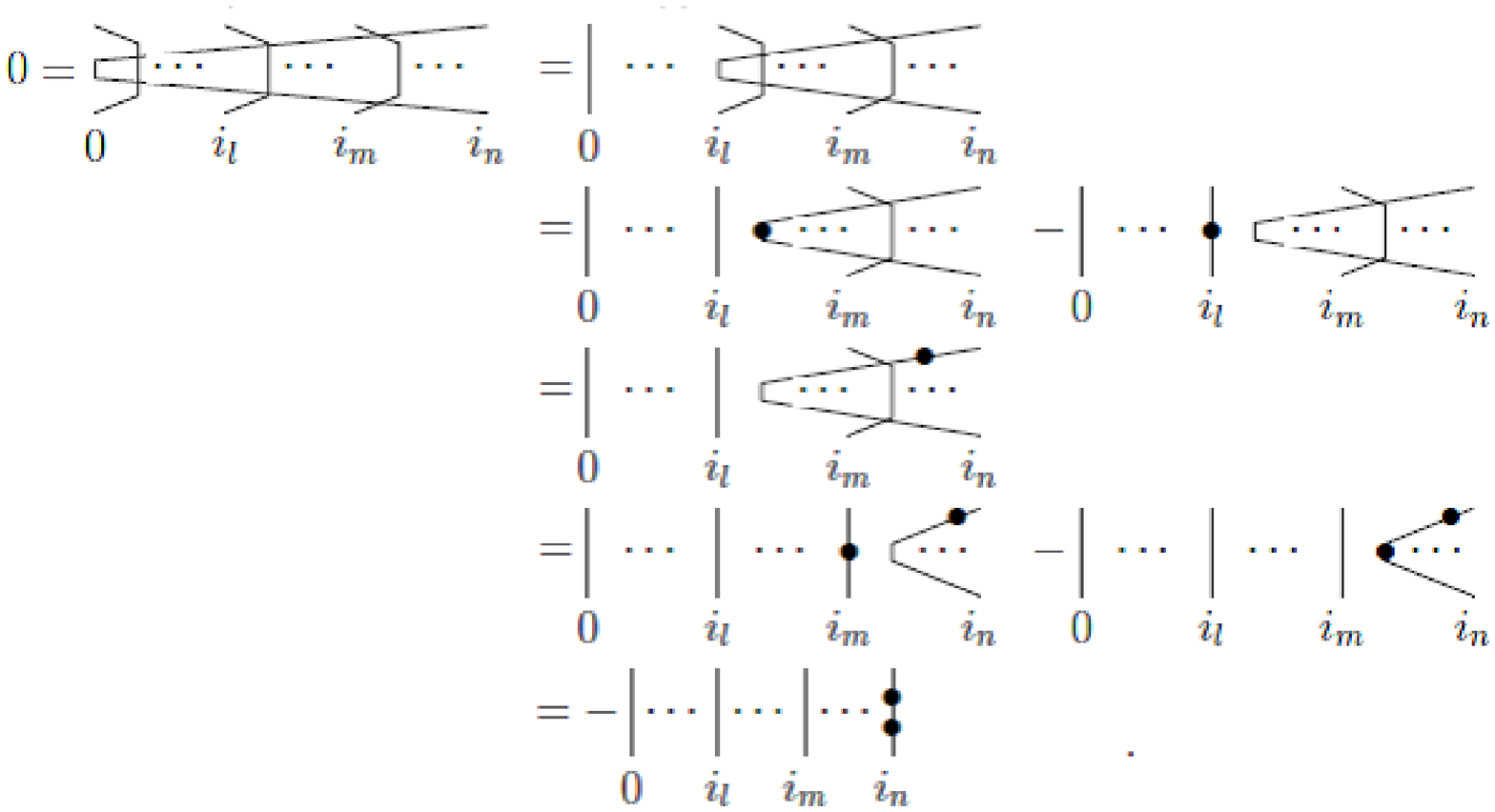}

Also we assume there $i_{l} \rightarrow i_{n}\rightarrow i_{m}$, 
but the difference with the case $i_{l} \leftarrow i_{n}\leftarrow i_{m}$ is only signs.
Therefore $(b)$ follows. 
\end{proof}

For $R_{n}$, set two subsets $I_{n}^{e}$, $I_{n}^{1}$ of $I_{n}$ as follows : 
\begin{eqnarray*}
I_{n}^{e}&=& \left\{ {\bf i}\in I_{n} \verb@|@ {\bf e}({\bf i})\not =0 \right\}\\
I_{n}^{1}&=& \left\{ {\bf i}\in I_{n}^{e} \verb@|@ i_{2}=1 \right\}
\end{eqnarray*}

And set an idempotent ${\bf e}$ of $R_{n}$ as follows : 
\begin{equation*}
{\bf e}=\displaystyle \sum_{ {\bf i} \in I_{n}^1 } {\bf e}({\bf i})
\end{equation*}

At last, set two maps \ $\hat{}:I_{n-1}^{e}(\alpha)\rightarrow I_{n}^{1}(\alpha)$, 
$\bar{}:I_{n}^{1}(\alpha)\rightarrow I_{n-1}^{e}(\alpha)$ as follows : 
\begin{eqnarray*}
\hat{{\bf i}}=(0,1,i_{2}+1,\cdots ,i_{n-1}+1)& for &{\bf i}=(0,i_{2},\cdots ,i_{n-1}),\\
\bar{{\bf i}}=(0,i_{3}-1,\cdots ,i_{n}-1)& for &{\bf i}=(0,1,i_{3},\cdots ,i_{n}).
\end{eqnarray*}
In other word, $\hat{}$ increments $i_{k}$ except $i_{1}$ and inserts $1$ at second, 
$\bar{}$ decrements $i_{k}$ except $i_{1}$ and remove $i_{2}$.
Both maps are bijection and inversion of the other. 

\begin{prop}
For each $n>2$, an isomorphism of algebras
\begin{equation*}
R_{n-1} \rightarrow {\bf e}R_{n}{\bf e}
\end{equation*}
is obtained as follows : 
\begin{equation*}
{\bf e}({\bf i})\mapsto {\bf e}(\hat{{\bf i}}) \ , \ 
y_{n-1} \mapsto y_{n} \ , \ 
\psi_{k} \mapsto \psi_{k+1} \ .
\end{equation*}
\end{prop}
\begin{proof}
For ${\bf e}({\bf i})$, ${\bf e}({\bf i})=0$ and ${\bf e}(\hat{{\bf i}})=0$ are equivalent. 
For $y_{k}$, what we check is only $y_{n-1} \in R_{n-1}$ and $y_{n} \in R_{n}$ by Prop.4, 
and $y_{n}$ and $\psi_{n-1}$ cannot appear at the same time. 
Then it is easy to check relations are preserved. 
Since elements in ${\bf e}R_{n}{\bf e}$ can be presented without $\psi_{1}$, 
we can make the inversion map ${\bf e}R_{n}{\bf e}\rightarrow R_{n-1}$ as follows : 
\begin{equation*}
{\bf e}({\bf i})\mapsto {\bf e}(\bar{{\bf i}}) \ , \ 
y_{n} \mapsto y_{n-1} \ , \ 
\psi_{k} \mapsto \psi_{k-1} \ .
\end{equation*}
\end{proof}

\section{Representation Theoretical Facts}

Using isomorphism given in [3], each $R_{n}$ is replaced by well-known object in representation theory. 
Using the facts in it, we complete the proofs of previous section. 

$R_{n}^{\Lambda}(\alpha )$ is a cyclotomic KLR algebra for general $\Lambda $, 
and $|\alpha|$ is a size of weight $\alpha $, but it is not necessary, we skip definitions. 

\begin{thm}[Brundan-Kleshchev,Rouquier] \ 
\begin{itemize}
\item[$(a)$] $\displaystyle \bigoplus _{|\alpha| =n} R_{n}^{\Lambda}(\alpha )\cong \mathcal{H}_{q}^{\Lambda}(n)$\\
The right side is Ariki-Koike algebra determined by $\Lambda $ and $n,q=\sqrt[n]{1} \in \mathbb{C}$. 
\item[$(b)$] $R_{n}^{\Lambda}(\alpha)$ is a block. That is, an indecomposable two-sided ideal. 
\end{itemize}
\end{thm}

We set $\Lambda =\Lambda _{0}$. 
In this case, Ariki-Koike algebra is Hecke algebra $H_{q}(\mathcal{S}_{r})$ of type $A$. 
The following theorem holds. For notations in the theorem, see Mathas([4] p.50 Ex.18). 

\begin{thm}[Dipper-James]
Let $\lambda$ be a partition of $r$. \\
There exists $H_{q}(\mathcal{S}_{r})$-module $S^{\lambda}$ with following properties : \\
Let $n$ be minimum integer satisfying $1+q+q^{2}+\cdots +q^{n-1}=0$. 
\begin{itemize}
\item[$(a)$] If $\lambda$ is $n$-$regular$ (the same number doesn't continue n times),
then {\rm top} of $S^{\lambda}$ is uniquely determined. 
In this case, we denote $D^{\lambda}$ for {\rm top}$S^{\lambda}$. 
\item[$(b)$] $\left\{ D^{\lambda} \mid \lambda : n {\mbox{-}}regular \right\}$ is 
complete list of simple $H_{q}(\mathcal{S}_{r})$-modules. 
\end{itemize}
\end{thm}

The following lemma holds in general.

\begin{lem}
Let $P^{\lambda}$ a indecomposable projective module corresponding to $D^{\lambda}$. \\
As a left module, 
\begin{equation*}
H_{q}(\mathcal{S}_{r}) \cong \displaystyle \bigoplus_{\lambda} ({\rm dim}D^{\lambda})P^{\lambda}
\end{equation*}
\end{lem}

The following property holds in this time [5].

\begin{thm}\label{thm:Uno}
As an element of Grothendieck group, 
\begin{itemize}
\item $\left[ D^{(n)} \right] = \left[ S^{(n)} \right]$
\item $\left[ D^{(n-k,1^{k})} \right] = -\left[ D^{(n-k+1,1^{k-1})} \right] + \left[ S^{(n-k,1^{k})} \right]$
\end{itemize}
\end{thm}

By using hook length formula, the following property holds. 

\begin{prop}\label{prp:hook}
\begin{equation*}
{\rm dim} S^{(n-k,1^{k})} = \binom{n-1}{k}
\end{equation*}
\end{prop}

By using Thm.9 and Prop.10, the following property holds. 

\begin{prop}
For $0\leq k\leq n-1$, denote $\lambda_{k} = (n-k,1^{k})$. 
\begin{equation*}
\displaystyle \sum_{k=0}^{n-1} {\rm dim} D^{\lambda_{k}} = 2^{n-2}
\end{equation*}
\end{prop}
\begin{proof}
Since ${\rm dim}D^{\lambda_{k}}=-{\rm dim}D^{\lambda_{k-1}}+{\rm dim}S^{\lambda_{k}}$, \\
we obtain ${\rm dim}D^{\lambda_{k}}+{\rm dim}D^{\lambda_{k-1}}={\rm dim}S^{\lambda_{k}}=\binom{n-1}{k}$. \\
Therefore if $n$ is odd, 
\begin{equation*}
\displaystyle \sum_{k=0}^{n-1} {\rm dim} D^{\lambda_{k}} = 1+\binom{n-1}{2}+\binom{n-1}{4}+\cdots +\binom{n-1}{n-1}=2^{n-2}
\end{equation*}
if even, 
\begin{equation*}
\displaystyle \sum_{k=0}^{n-1} {\rm dim} D^{\lambda_{k}} = \binom{n-1}{1}+\binom{n-1}{3}+\cdots +\binom{n-1}{n-1}=2^{n-2}
\end{equation*}
\end{proof}

Therefore we obtain the following corollary.

\begin{cor}
Every $2^{n-2}$ ${\bf e}({\bf i})$s obtained in Thm.3 is primitive idempotent. 
\end{cor}

The folloing property holds. 

\begin{prop}
If ${\bf e}({\bf i}) \neq 0$ then $y_{n}{\bf e}({\bf i}) \neq 0$.
\end{prop}
\begin{proof}
There are no elements except for $y_{n}{\bf e}({\bf i})$ in ${\bf e}({\bf i})H_{n}{\bf e}({\bf i})$ 
such that linearly independent to ${\bf e}({\bf i})$.
On the other hand, since there are no indecomposable simple projective modules by Thm.9, 
${\rm dim}({\rm End}({\bf e}({\bf i})H_{n}))\geq 2$. 
Hence $y_{n}{\bf e}({\bf i})\neq 0$ from ${\rm End}({\bf e}({\bf i})H_{n})\cong {\bf e}({\bf i})H_{n}{\bf e}({\bf i})$. 
\end{proof}

\section{Quivers and relations}

Assume $K$ is algebraic closed, then there is a fact ; \\ \hspace{3cm}
$R_{n}$ is Morita equivalent to Brauer tree algebra of $A_{n}$ type. 

\medskip

Hence we can interpret $R_{n}$ using a quiver and relations. 

Let $Q_{1}$ and $Q_{n} (n>1)$ be a quiver as follows;

\hspace{3cm}
\includegraphics[width=10cm,angle=0,keepaspectratio,clip]{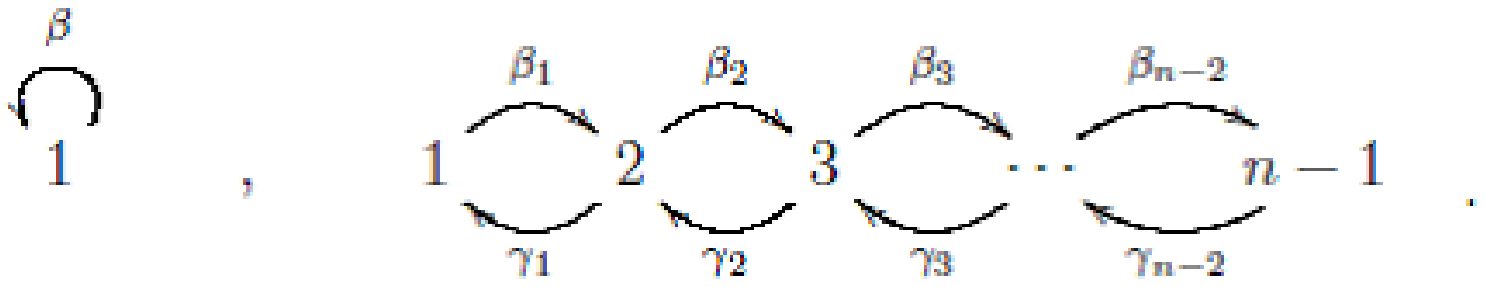}

And rerations be as follows; \\ \hspace{15mm}
$\beta^{2}, \beta_{i} \beta_{i+1}, \gamma_{i+1} \gamma_{i}, 
\gamma_{i} \beta_{i} - \beta_{i+1} \gamma_{i+1} (1\leq i \leq n-3).$

\medskip

Since $R_{2}$ has only one primitive idempotent ${\bf e}(0,1)$, 
it is corresponding to vertex 1 of the quiver, 
and $y_{2}{\bf e}(0,1)$ is corresponding to $\beta$. 
Since $y_{2}^{2}{\bf e}(0,1)=0$, it suffices the relation. 

Since the number of vertices of a quiver coinsides with
the number of isomorphic classes of indecomposable projective modules, 
we can classify $2^{n-2}$ primitive idempotents of $R_{n}$. 

\begin{prop}
For each $R_{n}$, the two indecomposable projective modules corresponding to 
two primitive idempotents ${\bf e}({\bf i})$ and ${\bf e}({\bf j})$ are isomorphic 
if and only if $i_{n}=j_{n}$. 

In particular, the isomorphic class of indecomposable projective modules has $(n-1)$ elements. 
\end{prop}
\begin{proof}
It is sufficient to decompose $2^{n-2}$ primitive idempotents to $(n-1)$ sets.
Let $i_{n}=k$ for ${\bf i} \in I^{e}_{n}$. 
We can get ${\bf i}_{k} := (0,1,2,\cdots ,k-1,n-1,n-2,\cdots ,k+1) \in I^{e}_{n}$ 
by swapping $i_{k}$ and $i_{k+1}$ for some $2 \leq k \leq n-2$ 
such that $i_{k} \not = i_{k}\pm 1$ repeatedly. 
Hence if ${\bf e}({\bf i}) \not =0$,${\bf e}({\bf j})\not =0$ and $i_{n}=j_{n}$, 
we can obtain ${\bf i}$ from ${\bf j}$ by swapping restricted as above. 
Then we obtain $R_{n}{\bf e}({\bf i}) \cong R_{n}{\bf e}({\bf j})$ by using [1] Prop.2.13. 
Hence there are at most $(n-1)$ isomorphic classes that is characterized by the last number. 
\end{proof}

Since we can take ${\bf e}({\bf i}_{k})$s as representative elements of isomorphic classes, 
each ${\bf e}({\bf i}_{k})$ is corresponding to the vertex $k$ of $Q_{n}$. 

We can easily specify (but be careful about signs) 
what are corresponding to arrows $\beta_{t}$, $\gamma_{t}$ 
with a notation $\ _{\bf i}1_{\bf j}$ from [1] in the proof of Thm.2.5; 
$\ _{\bf i_{t}}1_{\bf i_{t+1}}$, $(-1)^{t} \ _{\bf i_{t+1}}1_{\bf i_{t}}$ are corresponding to 
$\beta{t}$, $\gamma{t}$. 

We can easily confirm $(-1)^{t}\ _{\bf i_{t+1}}1_{\bf i_{t}} \ _{\bf i_{t}}1_{\bf i_{t+1}} = 
\ _{\bf i_{t+1}}1_{\bf i_{t+2}} (-1)^{t+1} \ _{\bf i_{t+2}}1_{\bf i_{t+1}} = 
(-1)^{t} y_{t}{\bf e}({\bf i}_{t+1})$. 
Notice that in $\ _{\bf i_{t}}1_{\bf i_{t+1}} \ _{\bf i_{t+1}}1_{\bf i_{t+2}}$, 
the strands colored by $t,t+1,t+2$ are twisted each other, 
and there must be $\psi_{m} \psi_{m+1} \psi_{m}$ which is never appeared in $R_{n}$; 
if it appears, three colors appears at mth which is never happen by Thm.3. 
Then $\ _{\bf i_{t}}1_{\bf i_{t+1}} \ _{\bf i_{t+1}}1_{\bf i_{t+2}}=0$, 
and similarly $\ _{\bf i_{t+2}}1_{\bf i_{t+1}} \ _{\bf i_{t+1}}1_{\bf i_{t}}=0$. 

\section{The dimension of $R_{n}$}

\begin{prop}
${\rm dim}_{K}R_{n}=\binom{2(n-1)}{n-1}$
\end{prop}
\begin{proof}
For $n=2$, there are only $2$ linearly independent elements; ${\bf e}_{(0,1)}$ and 
$y_{2}{\bf e}_{(0,1)}$. 

Note that there are $\binom{n-2}{k-1}$ elements in an isomorphic class consisting ${\bf i}_{k}$. 

As we have shown in the previos section, there are only $4$ types of elements 
corresponding to vertices, $\beta_{t}$, $\gamma_{t}$ and 
$\gamma_{t} \beta_{t}=\beta_{t+1} \gamma_{t+1}$. 

The only difference between the first and the last one is that 
there is a dot on the $n$th strand or not; 

\vspace{-2mm} \hspace{5cm}
\includegraphics[width=5cm,angle=0,keepaspectratio,clip]{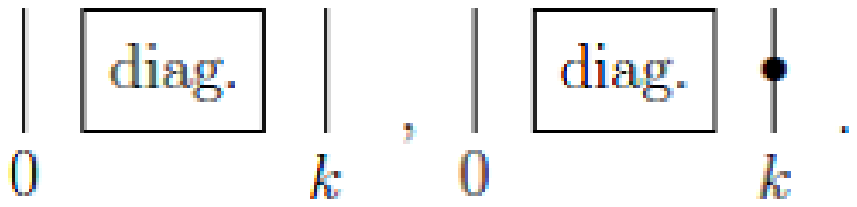}

And the shape of these diagrams is just $\ _{{\bf i}}1_{{\bf j}}$ 
such that $i_{n}=j_{n}=k$, hence there are 
$2 \displaystyle \sum_{k=1}^{n-1} \binom {n-2} {k-1} ^{2}$ diagrams. 

\medskip

Second one and third one has one to one correspondense by horizontal mirroring. 
Since $\ _{{\bf i}}1_{{\bf j}}\ (i_{n}=j_{n}\pm 1)$ has $\psi_{n-1}$, 
there are no dots on this diagram. 
Hence there are 
$2 \displaystyle \sum_{k=1}^{n-2} \binom {n-2} {k-1} \binom {n-2} {k}$ diagrams. 

Finally we obtain following equation; 

\begin{eqnarray*}
2 \displaystyle \sum_{k=1}^{n-1} \binom {n-2} {k-1} ^{2} +2 \displaystyle \sum_{k=1}^{n-2} \binom {n-2} {k-1} \binom {n-2} {k}
& = & 
\displaystyle \sum_{k=1}^{n-2} \left\{ \binom {n-2} {k-1} + \binom {n-2} {k} \right\}^{2} +2 \\
 & = & \displaystyle \sum_{k=1}^{n-2} \binom {n-1} {k} ^{2} +2 \\
 & = & \displaystyle \sum_{k=0}^{n-1} \binom {n-1} {k} ^{2} \\
 & = & \displaystyle \binom {2(n-1)} {n-1} .\\
\end{eqnarray*}

\end{proof}

\begin{ex}

For $n=3$, we have $\binom {4} {2} = 6$ diagrams as follows; 

\hspace{3cm}
\includegraphics[width=10cm,angle=0,keepaspectratio,clip]{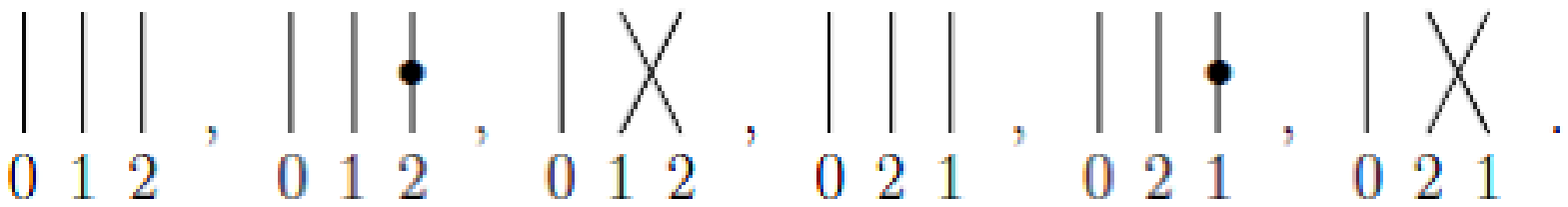}

\end{ex}

 
\ifx\undefined\bysame 
\newcommand{\bysame}{\leavevmode\hbox to3em{\hrulefill}\,} 
\fi

\end{document}